\documentclass[12pt]{article}
\usepackage[utf8]{inputenc}
\usepackage[english]{babel}
\usepackage[margin=2cm]{geometry}
\usepackage{fancyhdr}
\usepackage{nameref}
\usepackage{amsmath}
\usepackage{amssymb}
\usepackage{mathtools}
\usepackage{stmaryrd}
\usepackage{mleftright}
\usepackage{amsthm}
\usepackage{xspace}
\usepackage{comment}
\usepackage{tabularx}
\usepackage{enumitem}
\usepackage{tikz-cd} 
\usepackage[inline]{asymptote}
\usepackage{calc}
\usepackage{hyperref}
\usepackage[nameinlink]{cleveref}
\usepackage{ytableau} 
\newcommand\numbers{\textup{(\arabic*)}}

\newcommand{\qedd}{\hfill $\blacksquare$}

\DeclareMathOperator{\weight}{weight}

\DeclareMathAlphabet{\mathmybb}{U}{bbold}{m}{n}
\newcommand\ind{\mathmybb{1}}
\renewcommand\phi{\varphi}

\newcommand\RR{{\mathbb{R}}}

\newcommand\ZZ{{\mathbb{Z}}}
\newcommand\Dd{{\mathcal{D}}}

\newcommand\Ff{{\mathcal{F}}}

\newcommand\bzer{{\pmb{0}}}
\newcommand\bone{{\pmb{1}}}

\newcommand\bh{{\pmb{h}}}

\newcommand\bx{{\pmb{x}}}

\newcommand{\ctab}[1]{\vcenter{\hbox{#1}}}

\theoremstyle{plain}

\newtheorem{innercustomthm}{}
\newenvironment{repthm}[1]
  {\renewcommand\theinnercustomthm{#1}\innercustomthm}
  {\endinnercustomthm}

\newtheorem{theorem}{Theorem}[section]
\newtheorem*{theorem*}{Theorem}
\newtheorem{corollary}[theorem]{Corollary}
\newtheorem{lemma}[theorem]{Lemma}

\newtheorem*{lemma*}{Lemma}
\newtheorem{prop}[theorem]{Proposition}

\theoremstyle{definition}
\newtheorem{definition}[theorem]{Definition}
\newtheorem{example}[theorem]{Example}
\newtheorem{remark}[theorem]{Remark}

\crefname{claim}{claim}{claims}
\Crefname{claim}{Claim}{Claims}
\crefname{app-corollary}{corollary}{corollaries}
\Crefname{app-corollary}{Corollary}{Corollaries}
\crefname{app-definition}{definition}{definitions}
\Crefname{app-definition}{Definition}{Definitions}
\crefname{figure}{figure}{figures}
\Crefname{figure}{Figure}{Figures}
\crefname{lemma}{lemma}{lemmata}
\Crefname{lemma}{Lemma}{Lemmata}
\crefname{app-lemma}{lemma}{lemmata}
\Crefname{app-lemma}{Lemma}{Lemmata}
\crefname{prop}{proposition}{propositions}
\Crefname{prop}{Proposition}{Propositions}
\crefname{case}{Case}{Propositions}

\crefname{app-theorem}{theorem}{theorems}
\Crefname{app-theorem}{Theorem}{Theorems}
\newcommand\dc{\textit{d}-complete\xspace}
\newcommand\DOCTITLE{A generalization of RSK to \dc posets}
\newcommand\AUTHOR{Son Nguyen, Joseph Vulakh, Dora Woodruff}
\newcommand\DATE{ }
\makeatletter
\newcommand\currentname{\@currentlabelname} % the section name
\makeatother
\hypersetup{
    colorlinks=true,
    linkcolor=blue,
    filecolor=blue,      
    urlcolor=blue,
}

\AtEndDocument{%
  \par
  \bigskip
  \begin{tabular}{@{}l@{}}%
    \textsc{Department of Mathematics, MIT}\\
    \textit{E-mail address}: \texttt{sonnvt@mit.edu}
  \end{tabular}
  
  \par
  \bigskip
  \begin{tabular}{@{}l@{}}%
    \textsc{Department of Mathematics, MIT}\\
    \textit{E-mail address}: \texttt{jvulakh@mit.edu}
  \end{tabular}
  
  \par
  \bigskip
  \begin{tabular}{@{}l@{}}%
    \textsc{Department of Mathematics, MIT}\\
    \textit{E-mail address}: \texttt{dorawood@mit.edu}
  \end{tabular}
}
\fancypagestyle{nohead}{
  \fancyhf{} % remove everything
  \cfoot{\copyright\ \AUTHOR}}
\AtBeginDocument{%
  \mathchardef\mathcomma\mathcode`\,
  \mathcode`\,="8000 
}
{\catcode`,=\active
  \gdef,{\mathcomma\discretionary{}{}{}}
}

\title{\DOCTITLE}
\author{\AUTHOR}
\date{\DATE}
\begin{document}

\maketitle

\begin{abstract}
    The hook length formula for \dc posets expresses the number of linear extensions of a \dc poset $P$ in terms of hooks of $P$. It generalizes the usual hook length
    formula for standard Young tableaux, as well as hook length formulas for shifted Young tableaux and trees. We give a new proof of the hook length formula for \dc posets which is elementary and purely combinatorial. Our approach is to define a generalization of the
    Robinson-Schensted-Knuth bijection for \dc posets, which may
    be of independent interest.
\end{abstract}

\section{Introduction}

Given a poset $P$,
a natural question is to count the number of linear extensions of $P$.
A poset $P$ is said to have a \emph{hook length formula}
if the number of linear extensions has the form
$\frac{|P|!}{\prod_{x\in P}h_P(x)}$,
where $h_P$ is some function
from elements of $P$ to the positive integers
and $h_P(x)$ is called the \emph{hook length} of $x$.
A simple example is when $P$ is a rooted tree,
in which case the hook length of $x$ is simply the number of elements
nonstrictly dominated by $x$ in the tree.
Another example is when $P$ is a (rotation of a) Young diagram
of some partition $\lambda$.
In this case, the number of linear extensions of
$P$ is the same as the number of standard Young tableaux (SYT) of shape $\lambda$.
This is the celebrated hook length formula for SYTs.
Several generalizations of this hook length formula
(together with bijective proofs) have been given for shifted Young
diagrams, skew Young diagrams, skew shifted Young diagrams,
and other related posets; see
\cite{morales2018hook1,morales2017hook2,naruse2014schubert,
konvalinka2020hook, konvalinka2020bijective}.

One way to generalize the hook length formula is to find a more
general family of posets to which it applies.
There is a quite general family of posets, containing all
Young diagrams, shifted Young diagrams, and rooted trees, called \dc
posets. They were introduced by Proctor
\cite{proctor1999dynkin,proctor1999minuscule} and can be defined via a
small set of local axioms \cite{proctor2019d}. We refer the readers to
\cite[Section 2.2]{naruse2019skew} for a summary of the
connection between \dc posets and Weyl groups, which was their
original motivation.
By design, minuscule posets are \dc posets,
and Chaput--Perrin~\cite{chaput2009towards} used them to give a
positive combinatorial formula for computing certain
$\Lambda$-minuscule Schubert structure constants for general Kac-Moody
flag varieties. See also \cite{ilango2018unique} for a conjectural
$K$-theoretic version of Chaput--Perrin's formula. In addition, \dc
posets have the confluence property of jeu de taquin
similar to SYTs~\cite{proctor2009d}.

With all these nice properties, it is not surprising that \dc posets
also have a hook length formula,
described in \Cref{sec:defs}
and first proved by Proctor~\cite{proctor2014d}.
An additional proof was given by Kim--Yoo~\cite{kim2019hook}
and was the first completely detailed proof
in the literature,
previous proofs having only been sketched in conference proceedings.
Their proof is very intricate,
using the technique of $q$-integrals and verifying many cases by computer.
Finally, Naruse-Okada~\cite{naruse2019skew}
proved a general formula for skew \dc posets
via the machinery of equivariant $K$-theory.
In addition to generalizing to skew \dc posets,
their method yields a \emph{multivariate} generalization
of the hook length formula
which assigns each linear extension of a \dc poset $P$ a weight
and expresses the sum of weights over all linear extensions
in terms of certain hook polynomials.

The multivariate hook length formula for \dc posets
obtained by Naruse-Okada~\cite[Corollary~5.4]{naruse2019skew}
gives a combinatorial interpretation of
the colored hook length formula for root systems
of Nakada~\cite{nakada2008colored}.
The special case of Young diagrams
was independently obtained but not published by Postnikov;
the proof is given in Hopkins's notes~\cite{hopkins2014rsk}
and generalizes Pak's~\cite{pak2001hook} proof of the hook length formula
for Young diagrams.

Our goal is to give a simple, purely combinatorial proof of
the multivariate hook length formula for \dc posets.
To this end, we further generalize the proof of Pak~\cite{pak2001hook},
which considers the celebrated Robinson-Schensted-Knuth correspondence
as a volume-preserving bijection between two polytopes.
We extend this notion of RSK to \dc posets,
and this generalized RSK may be of independent interest.

Central to our argument is a generalization
of the natural notion of \textit{diagonals} from Young diagrams
to \dc posets;
these correspond to the \dc colorings
of Proctor~\cite[Proposition~8.6]{proctor1999minuscule}.
In \Cref{sec:defs}, we formally define diagonals
and denote by $D(q)$ the diagonal of an element $q$
of \dc poset $P$.
We also define a generalization of hooks to \dc posets,
where we view the hook $\bh_P(p)$ of an element $p$
as a vector
with integer coordinates $h_P^{(D)}(p)$ indexed by diagonals of $P$
summing to the hook length of $p$
used by Proctor~\cite{proctor2014d}.
With this notation, we will use RSK to prove the following theorem,
first obtained in the context of \dc posets by Naruse-Okada~\cite{naruse2019skew}.

\begin{repthm}{\Cref{thm:mainthm}}
    Let $P$ be a \dc poset.
    Consider a family of indeterminates $\{x_D\}_D$
    indexed by diagonals $D$
    of $P$.
    For each element $p$ of $P$,
    define the \emph{hook polynomial}
    $H_p(\bx) = \sum_D h_P^{(D)}(p) x_D$.
    Also, for each linear extension $T$ of $P$
    given by $p_1 > p_2 > \dots > p_n$,
    define the \emph{weight} of $T$
    to be the rational function
    with $\weight(T)^{-1} = \prod_{i = 1}^n \sum_{j = i}^n x_{D(p_i)}$.
    Then
    \[
        \sum_T \weight(T) = \frac1{\prod_{p \in P} H_p(\bx)},
    \]
    where the sum is over all linear extensions $T$ of $P$.
\end{repthm}

Setting all indeterminates to $1$ and multiplying both sides by $n!$
recovers Proctor's~\cite{proctor2014d} hook length formula.

In the following subsection, we outline the argument of the proof
and motivate our definition of RSK.
We then formalize our definitions in \Cref{sec:defs}
and prove preliminary lemmas about \dc posets in \Cref{sec:lemmas}.
We define RSK for \dc posets in \Cref{sec:rsk},
and give several of its properties in \Cref{sec:rsk_prop}.
Finally, in \Cref{sec:proof},
we prove the multivariable hook length formula.

\subsection*{Proof sketch of \Cref{thm:mainthm}}
\label{subsec:proof_sketch}

We now sketch our proof of the multivariate hook length formula.
Our strategy follows
Pak's geometric proof of the hook length formula in~\cite{pak2001hook}.

To each \dc poset $P$, we associate two polytopes in $\RR^{|P|}$. The
first polytope is $P_{\text{fillings}}$, given by $x_p \geq 0$ for all
$p \in P$ and $\sum_{p \in P}H_p(\bx)x_p \leq 1$. The
volume of $P_{\text{fillings}}$ can be shown to be
\[
    \frac1{n!} \prod_{p \in P} \frac1{H_p(\bx)}.
\]
The second polytope is $P_{\text{RPP}}$, given by $x_p \geq 0$ for all $p \in P$, $\sum_{p \in P} x_p \leq 1$, and $x_p\leq x_q$ if $p\geq_P q$. Here, RPP stands for \emph{reverse plane partition}. The volume of $P_{\text{RPP}}$ can be shown to be
\[
    \frac1{n!} \sum_T \weight(T),
\]
where the sum is over all linear extensions $T$ of $P$.

Thus, to show \Cref{thm:mainthm},
it suffices to show that $P_{\text{fillings}}$ and $P_{\text{RPP}}$
have the same volume.
For this, we will define a piecewise-linear,
volume-preserving maps that sends
$P_{\text{fillings}}$ to $P_{\text{RPP}}$.
We call this map RSK, for reason explained in
\Cref{subsec:rsk_sketch}.
Our RSK map can be constructed via a series of \emph{toggle operations},
each of which has determinant $\pm 1$
when viewed as a linear map
$\mathbb{R}^{|\lambda|} \to \mathbb{R}^{|\lambda|}$,
whence RSK is volume-preserving
(\Cref{prop:bijec_d_volume}).
It also follows
that the toggle operations are invertible.
Hence, it remains to show that RSK maps
$P_{\text{fillings}}$ into $P_{\text{RPP}}$.
This (\Cref{prop:bijec_d_hl}) is the main technical part of the proof.

\subsection*{Acknowledgements}

We thank Alex Postnikov
for providing many valuable suggestions.
We also thank Soichi Okada for helpful feedback
on an earlier version of this manuscript.
The second author was supported by MIT UROP.

\section{Definitions}
\label{sec:defs}

Proctor~\cite{proctor1999dynkin} defined \dc posets
in terms of simple posets denoted $d_k(1)$.

\begin{definition}[\cite{proctor1999dynkin}]
    \label{def:d_k(1)}
    Let $k \geq 3$.
    The poset $d_k(1)$ has $2k - 2$ elements,
    of which two are incomparable,
    $k - 2$ form a chain above the two incomparable elements,
    and $k - 2$ form a chain below the two incomparable elements
    (see \Cref{ex:neck}). 
\end{definition}

We will use the following terminology of Proctor~\cite{proctor1999dynkin, proctor2019d}. 
Let $P$ be a poset.
\begin{itemize}
    \item The two incomparable elements of $d_k(1)$ are called the
        \emph{side elements}. 
    \item The $k - 2$ elements above the side elements are called
        \emph{neck elements}.
        For $k \geq 4$,
        all but the lowest neck element are called
        \emph{strict neck elements}. 
    \item The $k-2$ elements below the side elements are called
        \emph{tail elements}.
        For $k \geq 4$, all but the highest tail element are called
        \emph{strict tail elements}.
    \item An interval $[p,q]$ of $P$ is called
        a \emph{$d_k$-interval} if it is isomorphic to $d_k(1)$.
        The side, neck, and tail elements of $[p, q]$
        are those that correspond respectively
        to side, neck, and tail elements of $d_k(1)$.
        We will often simplify terminology
        and refer simply to $d$-intervals when the value of $k$ does not matter.
    \item A subposet of $P$ is called \emph{$d_k^-$-convex}
        if it is convex and isomorphic to $d_k(1)\setminus \{\sup d_k(1)\}$.
        Again, we refer simply to $d^-$-convex sets
        when the value of $k$ does not matter.
    \item We call $d_3$-intervals \emph{diamonds}. 
\end{itemize}

\begin{example}
    \label{ex:neck}
    A $d_4$-interval is shown below.
    The neck elements are $q$ and $d$, with $q$ a strict neck element.
    The side elements are $a$ and $b$,
    and the tail elements are $c$ and $p$, with $p$ a strict tail element. The interval $[p,d]$ is a $d_4^-$-convex set.
    
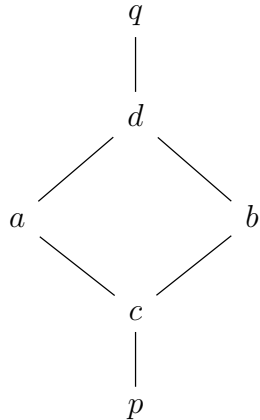
\begin{center}
        \begin{tikzcd}[ampersand replacement=\&]
            \& q \\
            \& d \\
            a \&\& b \\
            \& c \\
            \& p
            \arrow[no head, from=2-2, to=1-2]
            \arrow[no head, from=3-1, to=2-2]
            \arrow[no head, from=3-3, to=2-2]
            \arrow[no head, from=4-2, to=3-1]
            \arrow[no head, from=4-2, to=3-3]
            \arrow[no head, from=5-2, to=4-2]
        \end{tikzcd}
    \end{center}
\end{example}

A \dc poset is a poset which is built from $d_k$-intervals
in a `coherent way'.

\begin{definition}[\cite{proctor1999dynkin}]
    \label{def:d_comp}
    A finite poset $P$ is \emph{\dc} if it satisfies the following conditions:
    \begin{itemize}
        \item for each $d_k^-$-convex subset $I$ of $P$,
            there exists an element $z$ of $P$
            such that $I \cup \{z\}$ is a $d_k$-interval;
        \item for each $d$-interval $[w, z]$,
            the element $z$ does not cover any elements outside $[w, z]$;
        \item there are no $d^-$-convex sets
            which differ only in the minimal elements.
    \end{itemize}
\end{definition}

\begin{remark}
    Any rooted tree is trivially a \dc poset
    because it contains no $d^-$ convex sets.
    Young diagrams can be viewed as \dc posets:
    there is an element corresponding to each box,
    with one element covering another
    if one is directly above or to the left of the other.
    Shifted Young diagrams can be viewed as \dc posets
    in a similar way.
\end{remark}

\begin{example}
    \label{ex:d_comp}
    An example of a \dc poset which is not a tree,
    Young diagram, or shifted Young diagram is shown below.
    
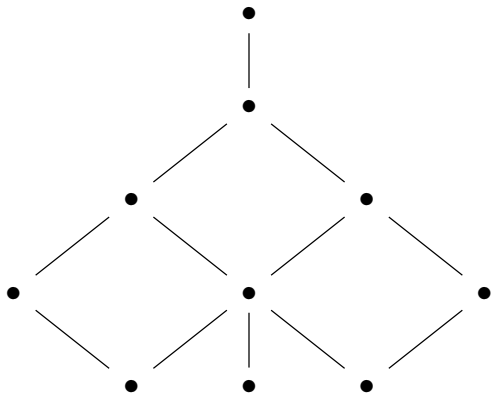
\begin{center}
        \begin{tikzcd}
            && \bullet \\
            && \bullet \\
            & \bullet && \bullet \\
            \bullet && \bullet && \bullet \\
            & \bullet & \bullet & \bullet \\
            \arrow[no head, from=1-3, to=2-3]
            \arrow[no head, from=2-3, to=3-2]
            \arrow[no head, from=2-3, to=3-4]
            \arrow[no head, from=3-2, to=4-1]
            \arrow[no head, from=3-2, to=4-3]
            \arrow[no head, from=3-4, to=4-3]
            \arrow[no head, from=3-4, to=4-5]
            \arrow[no head, from=4-1, to=5-2]
            \arrow[no head, from=4-3, to=5-2]
            \arrow[no head, from=4-3, to=5-4]
            \arrow[no head, from=4-5, to=5-4]
            \arrow[no head, from=4-3, to=5-3]
        \end{tikzcd}
    \end{center}
\end{example}

\begin{remark}
    An \emph{upper set} of a poset $P$ is a subposet $P'$ of $P$ such
    that if $x \in P'$ and $y \geq x$, then $y \in P'$.
    It follows from the axioms that upper sets of \dc posets
    are themselves \dc. The same is not true for lower sets, however. For example, a $d_k^-$-interval is a lower set of a $d_k$-interval, but it is not a \dc poset.
\end{remark}

The following generalization of the usual notion of diagonals in a Young diagram
will be central in our generalized hook length formula.

\begin{definition}
    \label{def:diag}
    Let $P$ be a \dc poset.
    Define an equivalence relation on the elements of $P$ as follows:
    for $p$, $q \in P$,
    let $p \sim q$
    if $[p, q]$ is a $d$-interval
    and extend the equivalence relation transitively.
    We call the equivalence classes under $\sim$ \emph{diagonals},
    and if $p \sim q$,
    we say that $p$ and $q$ are in the same diagonal.
    For an element $p$ of $P$,
    we let $D(p)$ denote the diagonal of $P$ containing $p$.
    Also, we let $\Dd(P)$ denote the set of diagonals of $P$.
\end{definition}

The partition of $P$ into its diagonals
corresponds to the coloring of $P$
obtained in~\cite[Proposition~8.6]{proctor1999minuscule}.

We also need a formal notion of diagonal adjacency.

\begin{definition}
    \label{def:diag_adj}
    We say two elements of $P$ are \emph{adjacent}
    if one element covers the other,
    and we say two diagonals $D_1$, $D_2$ of a \dc poset $P$ are adjacent
    if there exist $p_1 \in D_1$, $p_2 \in D_2$ which are adjacent.
\end{definition}

\begin{example}
    \label{ex:diag}
    The following diagram labels the diagonals
    of the \dc poset from \Cref{ex:d_comp}.
    The pairs of adjacent diagonals are
    $(D_1, D_2)$, $(D_2, D_3)$, $(D_2, D_4)$, $(D_3, D_5)$, $(D_4, D_6)$.
    
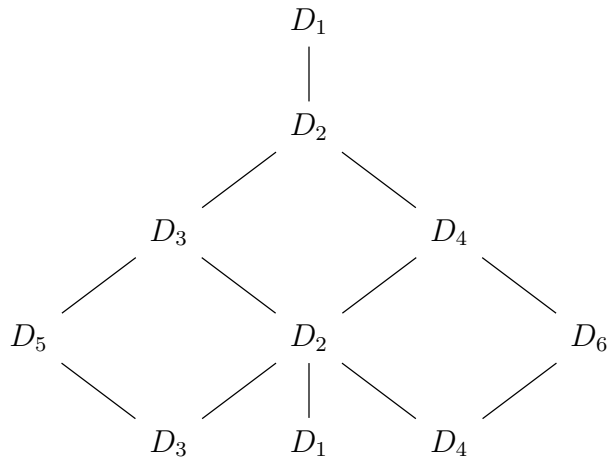
\begin{center}
        \begin{tikzcd}
            && D_1 \\
            && D_2 \\
            & D_3 && D_4 \\
            D_5 && D_2 && D_6 \\
            & D_3 & D_1 & D_4 \\
            \arrow[no head, from=1-3, to=2-3]
            \arrow[no head, from=2-3, to=3-2]
            \arrow[no head, from=2-3, to=3-4]
            \arrow[no head, from=3-2, to=4-1]
            \arrow[no head, from=3-2, to=4-3]
            \arrow[no head, from=3-4, to=4-3]
            \arrow[no head, from=3-4, to=4-5]
            \arrow[no head, from=4-1, to=5-2]
            \arrow[no head, from=4-3, to=5-2]
            \arrow[no head, from=4-3, to=5-4]
            \arrow[no head, from=4-5, to=5-4]
            \arrow[no head, from=4-3, to=5-3]
        \end{tikzcd}
    \end{center}
    The following diagram labels the diagonals
    of a \dc poset arising from a shifted Young diagram.
    
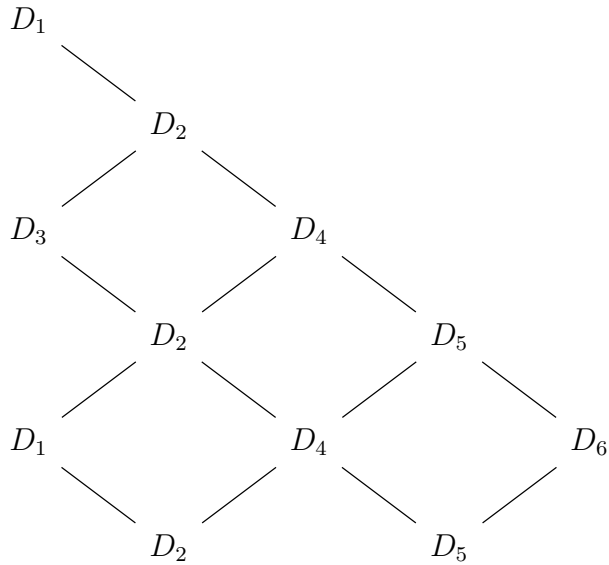
\begin{center}
        \begin{tikzcd}
            D_1 \\
            & D_2 \\
            D_3 && D_4 \\
            & D_2 && D_5 \\
            D_1 && D_4 && D_6 \\
            & D_2 && D_5 \\
            \arrow[no head, from=1-1, to=2-2]
            \arrow[no head, from=2-2, to=3-1]
            \arrow[no head, from=2-2, to=3-3]
            \arrow[no head, from=3-1, to=4-2]
            \arrow[no head, from=3-3, to=4-2]
            \arrow[no head, from=3-3, to=4-4]
            \arrow[no head, from=4-2, to=5-1]
            \arrow[no head, from=4-2, to=5-3]
            \arrow[no head, from=4-4, to=5-3]
            \arrow[no head, from=4-4, to=5-5]
            \arrow[no head, from=5-1, to=6-2]
            \arrow[no head, from=5-3, to=6-2]
            \arrow[no head, from=5-3, to=6-4]
            \arrow[no head, from=5-5, to=6-4]
        \end{tikzcd}
    \end{center}
    We observe that for shifted Young diagrams,
    most diagonals are as expected from Young diagrams.
    However, the leftmost elements alternate between two diagonals,
    rather than forming a single diagonal.
\end{example}

Finally, we need to formalize the notion of a hook vector in a \dc poset.

\begin{definition}
    \label{def:hook}
    For an element $p$ of a \dc poset $P$,
    we define the \emph{hook vector} of $p$,
    denoted $\bh_P(p) = [h_P^{(D)}(p)]_{D \in \Dd(P)}$,
    as an element of $\ZZ^{|\Dd(P)|}$ recursively as follows.
    If $p$ is not a neck element of any $d$-interval,
    then $h_P^{(D)}(p)$ is the number of elements of $D$
    nonstrictly dominated by $p$.
    If $p$ is the maximal element of some $d$-interval $[p', p]$
    with side elements $w$, $z$,
    then $\bh_P(p) = \bh_P(w) + \bh_P(z) - \bh_P(p')$.
    If the poset $P$ is clear from context,
    we will simply write $\bh(p)$.
    Also, for convenience, if $p \notin P$,
    we will set $\bh_P(p) = \bzer$.
\end{definition}

We will see in \Cref{cor:hook} that hook vectors are indeed well-defined.

\begin{example}
    \label{ex:hook}
    Consider the poset $d_4(1)$ with elements labeled as below.
    We will order the diagonals by $D(1)$, $D(2)$, $D(3)$, $D(4)$.
    
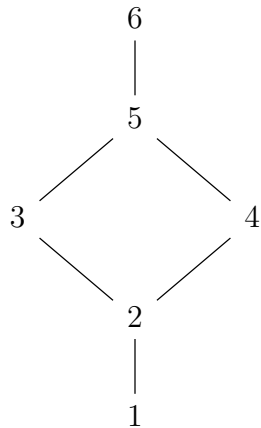
\begin{center}
        \begin{tikzcd}[ampersand replacement=\&]
            \& 6 \\
            \& 5 \\
            3 \& {} \& 4 \\
            \& 2 \\
            \& 1
            \arrow[no head, from=1-2, to=2-2]
            \arrow[no head, from=2-2, to=3-1]
            \arrow[no head, from=2-2, to=3-3]
            \arrow[no head, from=3-1, to=4-2]
            \arrow[no head, from=4-2, to=3-3]
            \arrow[no head, from=4-2, to=5-2]
        \end{tikzcd}
    \end{center}
    The hook vectors of elements $1$, $2$, $3$, $4$ are, respectively,
    $[1, 0, 0, 0]$, $[1, 1, 0, 0]$, $[1, 1, 1, 0]$, $[1, 1, 0, 1]$,
    because these elements are not top elements of any $d$-interval.
    Now, $5$ is the top element of the $d_3$-interval
    $[2,5]$ with side elements $3$ and $4$,
    so the hook vector of $5$ is $[1, 1, 1, 1]$.
    Finally, $6$ is the top element of $d_4(1)$ with side elements $3$ and $4$,
    so the hook vector of $6$ is $[1, 2, 1, 1]$.
\end{example}

To assist with calculations involving hook vectors,
we introduce some additional notation.

\begin{definition}
    For an element $p \in P$,
    let $\bone_p \in \ZZ^{|\Dd(P)|}$
    be given by $\bone_p = [\delta_{D(p), D}]_{D \in \Dd(P)}$.
    For a subset $S \subseteq P$,
    let $\bone_S = \sum_{p \in S} \bone_p$,
    which we call the \emph{indicator vector} of $S$.
\end{definition}

We observe that the sum of the entries of $\bh_P(p)$
is the hook length defined by Proctor~\cite{proctor2014d}.
Thus with this definition, Proctor's hook length formula is as follows.

\begin{theorem}[\cite{proctor2014d}]
    \label{thm:proctor}
    The number of linear extensions of a \dc poset $P$
    is given by $\frac{|P|!}{\prod_{p \in P} \sum_{D \in \Dd(P)} h_P^{(D)}(p)}$.
\end{theorem}

We prove the following generalization,
proved in the context of root systems
by Naruse-Okada~\cite[Corollary~5.4]{naruse2019skew}
and first stated for \dc posets by
Nakada~\cite{nakada2008colored}.

\begin{theorem}
    \label{thm:mainthm}
    Let $P$ be a \dc poset.
    Consider a family of indeterminates $\{x_D\}_{D \in \Dd(P)}$
    indexed by diagonals $D$
    of $P$.
    For each element $p$ of $P$,
    define the \emph{hook polynomial}
    $H_p(\bx) = \sum_{D \in \Dd(P)} h_P^{(D)}(p) x_D$.
    Also, for each linear extension $T$ of $P$
    given by $p_1 > p_2 > \dots > p_n$,
    define the \emph{weight} of $T$
    to be the rational function
    with $\weight(T)^{-1} = \prod_{i = 1}^n \sum_{j = i}^n x_{D(p_i)}$.
    Then
    \[
        \sum_T \weight(T) = \frac1{\prod_{p \in P} H_p(\bx)},
    \]
    where the sum is over all linear extensions $T$ of $P$.
\end{theorem}

\begin{remark}
    Setting all of the indeterminates $x_D$ in the theorem to $1$ recovers
    Proctor's original result as stated in \Cref{thm:proctor}.
\end{remark}

\begin{example}
    \label{ex:mainthm}
    Consider the poset $d_4(1)$ from \Cref{ex:neck,ex:hook}.
    In the diagram below we label elements by their diagonals.
    
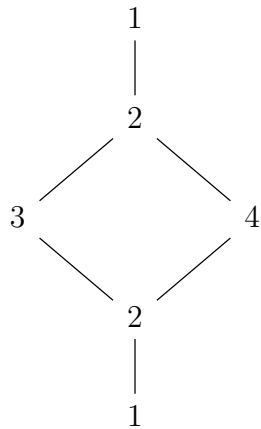
\begin{center}
        \begin{tikzcd}[ampersand replacement=\&]
            \& 1 \\
            \& 2 \\
            3 \& {} \& 4 \\
            \& 2 \\
            \& 1
            \arrow[no head, from=1-2, to=2-2]
            \arrow[no head, from=2-2, to=3-1]
            \arrow[no head, from=2-2, to=3-3]
            \arrow[no head, from=3-1, to=4-2]
            \arrow[no head, from=4-2, to=3-3]
            \arrow[no head, from=4-2, to=5-2]
        \end{tikzcd}
    \end{center}
    The poset $d_4(1)$ has two linear extensions,
    determined by a choice of ordering of the two side elements.
    This gives the sum of weights to be
    \begin{align*}
        &\tfrac1{x_1(x_1 + x_2)(x_1 + x_2 + x_3)(x_1 + x_2 + x_3 + x_4)
            (x_1 + x_2 + x_3 + x_4 + x_2)(x_1 + x_2 + x_3 + x_4 + x_2 + x_1)} \\
        &\phantom{=}+
        \tfrac1{x_1(x_1 + x_2)(x_1 + x_2 + x_4)(x_1 + x_2 + x_3 + x_4)
            (x_1 + x_2 + x_3 + x_4 + x_2)(x_1 + x_2 + x_3 + x_4 + x_2 + x_1)} \\
        &= \tfrac1{x_1(x_1 + x_2)(x_1 + x_2 + x_3 + x_4)
            (x_1 + 2x_2 + x_3 + x_4)(2x_1 + 2x_2 + x_3 + x_4)}
            \left( \tfrac1{x_1 + x_2 + x_3} + \tfrac1{x_1 + x_2 + x_4} \right)
            \\
        &= \tfrac{2x_2 + 2x_2 + x_3 + x_4}
            {x_1(x_1 + x_2)(x_1 + x_2 + x_3)(x_1 + x_2 + x_4)
            (x_1 + x_2 + x_3 + x_4)(x_1 + 2x_2 + x_3 + x_4)
            (2x_1 + 2x_2 + x_3 + x_4)} \\
        &= \tfrac1
            {x_1(x_1 + x_2)(x_1 + x_2 + x_3)(x_1 + x_2 + x_4)
            (x_1 + x_2 + x_3 + x_4)(x_1 + 2x_2 + x_3 + x_4)},
    \end{align*}
    which by \Cref{ex:hook} is the reciprocal of the product
    of the hook polynomials,
    demonstrating \Cref{thm:mainthm} for $d_k(4)$.
\end{example}

\section{Preliminary results}
\label{sec:lemmas}

Before defining RSK for \dc posets and proving
the multivariable hook length formula,
we give some simple structural facts about \dc posets.
In particular, we show that the hook vectors from \Cref{def:hook}
section are well-defined, and we show some useful facts about the
diagonals of \dc posets.
Throughout this section, let $P$ be a \dc poset.

\Cref{lem:forbid_subposet},
\Cref{prop:cover_bound}, and \Cref{prop:d_int_closure}
appear in Proctor's works~\cite{proctor1999dynkin,proctor1999minuscule};
for completeness, we give their proofs.

\begin{lemma}[{\cite[Section 3]{proctor1999dynkin}}]
    \label{lem:forbid_subposet}
    The poset $P$ cannot contain a subset
    $\{p_1, p_2, p_3, q_1, q_2, q_3\}$
    such that for $i \in [3]$,
    the sets $\{p_i, q_{i + 1}, q_{i + 2}\}$
    with indices taken $\bmod\ 3$
    are $d_3^-$-convex,
    as in the diagram below.

    
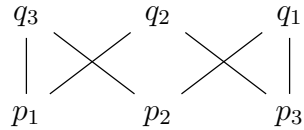
\begin{center}
        \begin{tikzcd}[ampersand replacement=\&]
            q_3 \& q_2 \& q_1 \\
            p_1 \& p_2 \& p_3
            \arrow[no head, from=2-1, to=1-1]
            \arrow[no head, from=2-1, to=1-2]
            \arrow[no head, from=2-2, to=1-1]
            \arrow[no head, from=2-2, to=1-3]
            \arrow[no head, from=2-3, to=1-2]
            \arrow[no head, from=2-3, to=1-3]
        \end{tikzcd}
    \end{center}
\end{lemma}

\begin{proof}
    Suppose for the sake of contradiction that such a subset does exist,
    and let $\{p_1, p_2, p_3, q_1, q_2, q_3\}$
    be one with maximal $q_1$.
    Because $P$ is \dc,
    there exist $r_1$, $r_2$, $r_3$ in $P$
    such that for each $i \in [3]$,
    the set $\{p_i, r_i, q_{i + 1}, q_{i + 2}\}$
    is a $d_3$-interval, as shown below.
    
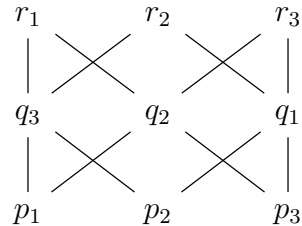
\begin{center}
        \begin{tikzcd}[ampersand replacement=\&]
            r_1 \& r_2 \& r_3 \\
            q_3 \& q_2 \& q_1 \\
            p_1 \& p_2 \& p_3
            \arrow[no head, from=2-1, to=1-1]
            \arrow[no head, from=2-1, to=1-2]
            \arrow[no head, from=2-2, to=1-1]
            \arrow[no head, from=2-2, to=1-3]
            \arrow[no head, from=2-3, to=1-2]
            \arrow[no head, from=2-3, to=1-3]
            \arrow[no head, from=3-1, to=2-1]
            \arrow[no head, from=3-1, to=2-2]
            \arrow[no head, from=3-2, to=2-1]
            \arrow[no head, from=3-2, to=2-3]
            \arrow[no head, from=3-3, to=2-2]
            \arrow[no head, from=3-3, to=2-3]
        \end{tikzcd}
    \end{center}
    No $r_i$ may cover an element outside its respective $d_3$-interval,
    so in particular $r_i$ cannot cover $q_i$.
    Thus, the elements $r_1$, $r_2$, and $r_3$ are distinct.
    Also,
    because each set $\{p_i, r_i, q_{i + 1}, q_{i + 2}\}$
    is an interval,
    there are no elements of $P$
    between $q_i$ and $r_j$ for $i \neq j$.
    Thus, for each $i$, the set $\{q_i, r_{i + 1}, r_{i + 2}\}$
    is convex.
    It follows that the set $\{q_1, q_2, q_3, r_1, r_2, r_3\}$
    is a set satisfying the conditions of the lemma,
    which since $r_3 \geq q_1$ contradicts maximality.
\end{proof}

\begin{prop}[{\cite[Section 3]{proctor1999dynkin}}]
    \label{prop:cover_bound}
    Each element of a \dc poset is covered by at most two elements.
\end{prop}

\begin{proof}
    Suppose element $p$ is covered by more than two elements,
    and let $p_1$, $p_2$, $p_3$ be elements covering $p$.
    Since they cover $p$, the elements $p_1$, $p_2$, and $p_3$ are incomparable.
    Thus for each $i \in [3]$ the set $\{p, p_{i + 1}, p_{i + 2}\}$
    is $d_3^-$-convex,
    where we take indices $\bmod 3$.
    Thus there exist elements $q_1$, $q_2$, and $q_3$
    such that  $\{p, p_{i + 1}, p_{i + 2}, q_i\}$ is a $d_3$-interval
    for $i \in [3]$.
    
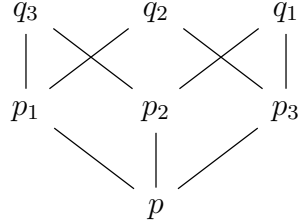
\begin{center}
        \begin{tikzcd}[ampersand replacement=\&]
            q_3 \& q_2 \& q_1 \\
            p_1 \& p_2 \& p_3 \\
            \& p
            \arrow[no head, from=2-1, to=1-1]
            \arrow[no head, from=2-1, to=1-2]
            \arrow[no head, from=2-2, to=1-1]
            \arrow[no head, from=2-2, to=1-3]
            \arrow[no head, from=2-3, to=1-2]
            \arrow[no head, from=2-3, to=1-3]
            \arrow[no head, from=2-1, to=3-2]
            \arrow[no head, from=2-2, to=3-2]
            \arrow[no head, from=2-3, to=3-2]
        \end{tikzcd}
    \end{center}
    For $i$, $j \in [3]$ with $i \neq j$,
    the element $q_i$ covers $p_j$.
    Thus for $i \in [3]$ the set $\{p_i, q_{i + 1}, q_{i + 2}\}$
    is $d_3^-$-convex,
    contradicting \Cref{lem:forbid_subposet}.
\end{proof}

\begin{prop}[{\cite[Lemma~8.3]{proctor1999minuscule}}]
    \label{prop:d_int_closure}
    A neck element of a $d$-interval
    does not cover any elements outside the $d$-interval,
    and a tail element of a $d$-interval
    is not covered by any elements outside the $d$-interval.
\end{prop}

\begin{proof}
    By \Cref{def:d_comp}, a neck element of a $d$-interval
    does not cover any elements outside the $d$-interval.
    Now suppose for the sake of contradiction that
    there exists a tail element of some $d$-interval $I$ covered by some $x \notin I$.
    Let $p$ be a maximal such element in $P$,
    and let $p$ be the tail element of $d$-interval $I$
    with side elements $w$ and $z$,
    and let $x$ be an element not in $I$ covering $p$.
    By \Cref{prop:cover_bound}, $p$ must be a strict tail element of $I$
    since otherwise $p$ would be covered by $w$, $z$, and $x$.
    Thus, there is a unique element $p'$ of $I$ covering $p$,
    and the elements $\{p, p', x\}$ form a $d_3^-$-convex set.
    Thus there must exist an element $x'$ covering $p'$ and $q$,
    and $x'$ cannot be in $I$ because $I$ is convex and $x \notin I$.
    This contradicts the maximality of $p$.
\end{proof}

\begin{prop}
    \label{prop:d_int_unique}
    If an element $p$ of $P$ is the maximal element of some $d$-interval,
    then $p$ is the maximal element of a \emph{unique} $d$-interval $I$,
    and any $d$-interval having $p$ as a neck element
    contains $I$.
    Similarly, if $p$ is the minimal element of some $d$-interval,
    then $p$ is the minimal element of a \emph{unique} $d$-interval $I$,
    and any $d$-interval having $p$ as a tail element
    contains $I$.
\end{prop}

\begin{proof}
    Suppose, for the sake of contradiction,
    that $P$ contains some element which is maximal in multiple $d$-intervals,
    and let $q$ be minimal among such elements.
    Let $[p, q]$ and $[p', q]$ be two distinct $d$-intervals.
    By \Cref{def:d_comp},
    the intervals $[p, q]$ and $[p', q]$ cannot differ
    in only their minimal elements,
    so we have $[p, q] \setminus \{p\} \neq [p', q] \setminus \{p'\}$.
    Also by \Cref{def:d_comp},
    the elements of $P$ covered by $q$
    must be in both $[p, q]$ and $[p', q]$.
    In particular, $q$ is a strict neck element of $[p, q]$
    if and only if it is a strict neck element of $[p', q]$;
    in this case, letting $q'$ be the unique element covered by $q$,
    we have distinct $d$-intervals
    $[p, q'] \setminus \{p\}$ and $[p', q'] \setminus \{p'\}$,
    contradicting the minimality of $q$.
    Thus $[p, q]$ and $[p', q]$ are diamonds
    differing only in their minimal elements,
    contradicting \Cref{def:d_comp}.
    The statement about $d$-intervals containing $q$ as a neck element follows;
    such intervals contain a $d$-interval
    with $q$ as maximal neck element,
    and thus contain the unique such interval.

    Now suppose, for the sake of contradiction,
    that $P$ contains elements which are minimal in multiple $d$-intervals,
    and let $p$ be maximal among such elements.
    Let $[p, q]$ and $[p, q']$ be two distinct $d$-intervals.
    By \Cref{prop:d_int_closure},
    the elements of $P$ covering $p$
    are contained in both $[p, q]$ and $[p, q']$.
    In particular, $p$ is a strict tail element of $[p, q]$
    if and only if it is a strict tail element of $[p, q']$,
    in which case letting $p'$ be the unique element covering $p$
    gives distinct $d$-intervals
    $[p', q] \setminus \{q\}$ and $[p', q'] \setminus \{q'\}$,
    contradicting the maximality of $p$.
    Thus $[p, q]$ and $[p, q']$ are diamonds
    sharing side elements $w$ and $z$.
    Then $\{w, q, q'\}$ and $\{z, q, q'\}$ are $d_3^-$-convex sets
    differing only in their minimal elements,
    contradicting \Cref{def:d_comp}.
    As before,
    the statement about $d$-intervals containing $p$ as a neck element follows.
\end{proof}

In light of \Cref{prop:d_int_unique}, we can make the following definition: 
\begin{definition}
    Let $p$ be an element of \dc poset $P$.
    We define $p^\uparrow$ to be the unique element such that
    $[p, p^\uparrow]$ is a $d$-interval, if such $p^\uparrow$ exists. 
    Similarly, we define $p^\downarrow$ to be the unique element such that
    $[p^\downarrow, p]$ is a $d$-interval, if such $p^\downarrow$ exists.
    By definition,
    $p$ is in the same diagonal as $p^\uparrow$ and $p^\downarrow$.
\end{definition}

As a consequence of \Cref{prop:d_int_unique},
we see that hook vectors are well-defined.

\begin{corollary}
    \label{cor:hook}
    The notion of hook vectors given in \Cref{def:hook} is well-defined.
\end{corollary}

\begin{proof}
    If $p$ is not a neck element of a $d$-interval,
    then $\bh_P(p)$ is clearly well-defined.
    If $p$ is a neck element of some $d$-interval,
    then by \Cref{prop:d_int_unique}
    there is a unique $p' \in P$
    such that $[p', p]$ is a $d$-interval,
    so the elements $p'$, $w$, and $z$ in \Cref{def:hook}
    are uniquely determined.
\end{proof}

Finally, we collect several facts about the diagonals of a \dc $P$. 

\begin{prop}
    \label{prop:diag_props}
    The following statements hold for diagonals of $P$:
    \begin{enumerate}[label=\numbers]
        \item \label{dprop:chain}
            The elements of a diagonal $D$ form a chain
            under the relation of $P$,
            with all covering relations given by $p \lessdot p^\uparrow$;
        \item Elements of a single diagonal $D$ cannot be adjacent;
            \label{dprop:not_adj}
        \item If $P'$ is an upper set of $P$,
            then two elements of $P'$ are in the same diagonal of $P'$
            if and only if they are in the same diagonal of $P$.
            Thus we may identify diagonals of $P$ with diagonals of $P'$;
            \label{dprop:up_id}
        \item If $C$, $D$ are adjacent diagonals
            such that the minimal element of $C$ is minimal in $P$,
            then each element of $D$ is adjacent
            to an element of $C$;
            \label{dprop:adj}
        \item Two diagonals having nonempty intersection
            with upper set $P'$
            are adjacent in $P'$ if and only if they are adjacent in $P$;
            \label{dprop:up_adj}
        \item If $C$, $D$ are adjacent diagonals of $P$
            with minimal elements $c$ and $d$, respectively,
            then at most one of $c$ and $d$ is minimal in $P$.
            \label{dprop:min}
    \end{enumerate}
\end{prop}

\begin{proof}
    \begin{enumerate}[label=\numbers]
        \item This follows from \Cref{prop:d_int_unique}.
        \item Suppose for the sake of contradiction that $p$, $q$
            are elements of $D$
            with $q$ covering $p$ in $P$.
            Then $q$ also covers $p$ in the subposet $D$,
            so, by statement~\ref{dprop:chain},
            $[p, q]$ is a $d$-interval of two elements,
            a contradiction.
        \item It is immediate from the definition that if two elements of $P'$
            are in the same diagonal of $P'$,
            then they are in the same diagonal of $P$,
            so we only prove the converse.
            Let $p < q$ be two elements of $P'$ in the same diagonal of $P$,
            so that by statement~\ref{dprop:chain}
            there exists a sequence $p = p_0 < \dots < p_n = q$
            of elements of $P$
            such that each $[p_{i - 1}, p_i]$ is a $d$-interval.
            Each element of each $[p_{i - 1}, p_i]$ dominates $p$ in $P$,
            and thus is in $P'$,
            so $p$ and $q$ are in the same diagonal of $P'$.
        \item Suppose for the sake of contradiction that not all elements of $D$
            are adjacent to an element of $C$.
            Let $d$ be the minimal element of $D$ that is not.

            \textit{Claim:} $d$ is minimal in $D$.
            Suppose for the sake of contradiction that $d^\downarrow$ exists.
            Then $d^\downarrow$ is adjacent to some element $c$ of $C$
            by assumption of minimality.
            We consider the possible positions of $c$
            with respect to the $d$-interval $[d^\downarrow, d]$.
            
            If $c$ is a side element of $[d^\downarrow, d]$,
            then $d$ is adjacent to $c$, a contradiction.
            If $c$ is contained in $[d^\downarrow, d]$
            but is not a side element,
            then $c$ is a tail element dominating $d^\downarrow$,
            and so $c^\uparrow$ exists and is dominated by $d$.
            In this case $d$ is adjacent to $c^\uparrow$,
            again a contradiction.
            
            Finally, if $c$ is not in $[d^\downarrow, d]$,
            then $d^\downarrow$ must dominate $c$ because
            by \Cref{prop:d_int_closure} elements outside $[d^\downarrow, d]$
            cannot dominate $d^\downarrow$.
            If $S = [d^\downarrow, d] \sqcup \{c\}$ is $d^-$-convex,
            then $c^\uparrow$ exists and dominates $d$,
            whence again $d$ is adjacent to $c^\uparrow$, a contradiction.
            On the other hand, if $S$ is not $d^-$-convex,
            then there is some element $x$ outside of $S$
            covering $c$ and dominated by an element of $S$.
            Now $\{c, d^\downarrow, x\}$ is $d_3^-$-convex,
            so there exists $y$ covering $d^\downarrow$ and $x$ such that
            $[c, y]$ is a $d_3$-interval.
            Now since $y$ covers $d^\downarrow$,
            by \Cref{prop:d_int_closure} it must be in $[d^\downarrow, d]$
            and thus either $y$ or $y^\uparrow$ is covered by $d$,
            depending on whether $y$ is a side element or a tail element.
            Also, $y$ is in $C$, and thus $y^\uparrow$ is as well,
            and so again $d$ is adjacent to an element of $C$,
            a contradiction.
            This finishes the proof that $d$ is minimal in $D$.
            \qedd

            Now let $d'$, $c'$ be adjacent elements of $D$ and $C$, respectively,
            such that $(d', c')$ is minimal in the product poset $P \times P$.
            Such a pair exists because $D$ and $C$ are adjacent diagonals.
            We have $d' \neq d$
            and $d$ is minimal in $D$,
            so the element $d'^\downarrow$ exists.
            By assumption of minimality of $(d', c')$,
            the element $d'^\downarrow$ cannot be adjacent to $c'$.
            Since $c'$ is adjacent to $d'$ but not $d'^\downarrow$,
            the element $c'$ is either a neck element of $[d'^\downarrow, d']$
            or covers $d'$.
            Thus if $c'^\downarrow$ exists,
            then by \Cref{prop:d_int_closure,prop:d_int_unique}
            $[c'^\downarrow, c']$ either is contained in or contains
            $[d'^\downarrow, d']$,
            whence $c'^\downarrow$ is adjacent to $d'^\downarrow$,
            contradicting the minimality of $(d', c')$.
            So $c'$ is minimal in $C$, and thus in $P$
            by the assumption on $C$.

            It follows that $d'$ covers a minimal element in $P$,
            and thus is minimal in $D$, a contradiction.
        \item Let the two diagonals be $C$ and $D$.
            If $C$ and $D$ are adjacent in $P'$,
            then by definition they are adjacent in $P$,
            so we have only to prove the other direction.
            Suppose $C$ and $D$ are adjacent in $P$.
            
            Consider the upper set $P''$ consisting of all elements of $P$
            which dominate some element of $C \cap P'$ or $D \cap P'$.
            The set $P' \setminus P''$ cannot contain any elements
            of $C$ or $D$.
            Also, any minimal element of $P''$ is in $C$ or $D$.
            Without loss of generality, let an element of $C$ be minimal in $P''$.
            Then by \ref{dprop:adj}, any element $d$ of $D \cap P'$ is adjacent
            to an element $c$ of $C \cap P'$.
            Suppose $c \notin P'$,
            in which case we must have $d \gtrdot c$.
            Since by assumption $P'$ contains elements of $C$,
            there must exist elements of $C$ dominating $c$.
            Thus $c^\uparrow$ exists,
            and must dominate $d$,
            whence $c^\uparrow \in P'$.
            Also, by \Cref{prop:d_int_closure},
            we have $d \in [c, c^\uparrow]$.
            If $[c, c^\uparrow]$ is a diamond,
            then $c^\uparrow$ is adjacent to $d$.
            Otherwise, $c^\uparrow \gtrdot d^\uparrow$.
            Either way, $C$ and $D$ are adjacent in $P'$.
        \item This follows from statement \ref{dprop:adj};
            if $c$ and $d$ were both minimal,
            each would have to be adjacent to an element
            of the other diagonal.
            Since an element covering a minimal element of $P$
            cannot be in the neck of a $d$-interval
            and thus is minimal in its diagonal,
            this implies that one of $c$ and $d$ is not minimal in $P$,
            a contradiction.
    \end{enumerate}
\end{proof}

\section{Defining RSK for \dc posets}
\label{sec:rsk}

As described in \Cref{subsec:rsk_sketch},
Pak~\cite{pak2001hook} conceives of RSK as a bijection from fillings
of Young diagrams by nonnegative real numbers to fillings of Young
diagrams with nonnegative real numbers that increase from left to
right and top to bottom. We will define RSK for \dc posets
similarly.

\begin{definition}
    \label{def:filling}
    A \emph{filling} of $P$ is a vector $[t_p]_{p \in P}$
    of nonnegative reals indexed by the elements of $P$.
    We say that a filling
    $[s_p]_{p \in P}$ is \emph{order-reversing}
    if $p \leq q$ implies that $s_p \geq s_q$.
    We will denote the set of fillings of $P$ by $\Ff_P$
    and the set of order-reversing fillings of $P$ by $\Ff'_P$.
\end{definition}

RSK can be described as a composition of certain \emph{toggle
operations}, first defined by Berenstein and Kirillov~\cite{kirillov_berenstein}.
We define toggle operations in the \dc setting.

\begin{definition}[Toggle operations]
    Let $[s_p]_{p \in P}$ be a filling of a \dc poset $P$.
    \emph{Toggling} at an element $p \in P$ is defined as follows.
    Of the elements covering $p$, let $x$ be such that $s_x$ is maximal;
    of the elements covered by $p$, let $y$ be such that $s_y$ is minimal. 
    If $p$ is not covered by any elements, we set $s_x = 0$,
    and similarly for $s_y$.
    We then change the label of $p$ to $s_p' = s_x + s_y - s_p$.
\end{definition}

When $P$ corresponds to a Young diagram, toggle operations are exactly
the toggle operations defined by Berenstein and Kirillov on Young
tableaux.

\begin{example}
    Consider a toggle operation at the bolded central element
    of the filling below.
    
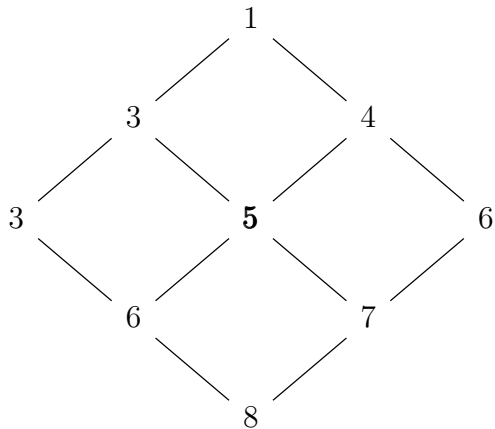
\begin{center}
        \begin{tikzcd}
            && 1 \\
            & 3 && 4 \\
            3 && \pmb{5} && 6 \\
            & 6 && 7 \\
            && 8
            \arrow[no head, from=1-3, to=2-2]
            \arrow[no head, from=1-3, to=2-4]
            \arrow[no head, from=2-2, to=3-1]
            \arrow[no head, from=2-2, to=3-3]
            \arrow[no head, from=2-4, to=3-3]
            \arrow[no head, from=2-4, to=3-5]
            \arrow[no head, from=3-1, to=4-2]
            \arrow[no head, from=3-3, to=4-2]
            \arrow[no head, from=3-3, to=4-4]
            \arrow[no head, from=3-5, to=4-4]
            \arrow[no head, from=4-2, to=5-3]
            \arrow[no head, from=4-4, to=5-3]
            \end{tikzcd}
    \end{center}
    We have $s_x = 4$, $s_y = 6$,
    and $s'_p = 4 + 7 - 5 = 6$.
\end{example}

\newcommand\RSK{\textsf{RSK}}

We now define $\RSK_P$ as a bijection
from fillings $\Ff_P$ to order-reversing fillings $\Ff'_P$.

\begin{definition}[RSK]
   Let $[t_p]_{p \in P}$ be a filling of a \dc $P$.
   We define $\RSK_P([t_p]_{p \in P})$ recursively.
   If $P$ has only one element, $\RSK_P$ is the identity map.
   Otherwise, pick a linear extension $p_{1} > p_{2} \dots > p_{n}$ of $P$,
   and let $P^{(j)} = \{p_i : i\leq j\}$.
   Suppose inductively that we have defined
   $\RSK_{P^{(j)}}([t_p]_{p \in P^{(j)}}) = [s_p]_{p \in P^{(j)}}$;
   we wish to define
   $\RSK_{P^{(j + 1)}}([t_p]_{p \in P^{(j + 1)}})$.
   To do so, we first label $p_i$ for $i \leq j$ with $s_{p_i}$
   and $p_{j+1}$ with $-t_{p_{j+1}}$.
   Then, we toggle every element in the diagonal of $p_{j+1}$
   (because by \Cref{prop:diag_props}\ref{dprop:not_adj}
       elements of $D(p_{j + 1})$ are not adjacent,
   this is well-defined without specifying an order for the toggles).
   The result is $\RSK_{P^{(j + 1)}}([t_p]_{p \in P^{(j + 1)}})$,
   and we can continue until we have defined $\RSK_P([t_p]_{p \in P})$. 
\end{definition}

This definition relies on a choice of linear extension of $P$,
but we will show shortly (\Cref{prop:rsk_well_def})
that this choice does not matter. 

\begin{example}
    Let us consider how RSK acts on the following filling
    with the left neck element preceding the right neck element
    in the linear ordering: 
    \begin{center}
        \begin{tikzcd}
            & 1 \\
            & 2 \\
            3 && 4 \\
            & 2 \\
            & 2
            \arrow[no head, from=1-2, to=2-2]
            \arrow[no head, from=2-2, to=3-3]
            \arrow[no head, from=3-1, to=2-2]
            \arrow[no head, from=3-1, to=4-2]
            \arrow[no head, from=4-2, to=3-3]
            \arrow[no head, from=4-2, to=5-2]
            \end{tikzcd}
    \end{center}
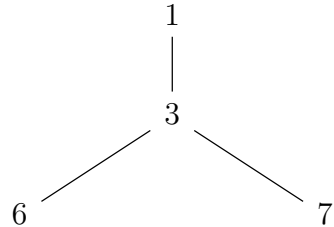
    \begin{center}
        \begin{tikzcd}
            {\text{Step 1}} && {\text{Step 2}}
            && {\text{Step 3}} && {\text{Step 4}} && {} \\
            1 && 1 && 1 && 1 \\
            && 3 && 3 && 3 \\
            &&& 6 && 6 && 7
            \arrow[no head, from=2-3, to=3-3]
            \arrow[no head, from=2-5, to=3-5]
            \arrow[no head, from=2-7, to=3-7]
            \arrow[no head, from=3-5, to=4-4]
            \arrow[no head, from=3-7, to=4-6]
            \arrow[no head, from=3-7, to=4-8]
        \end{tikzcd}
    \end{center}
    Considering $P^{(i)}$ for $i \leq 4$ is simple, because each element is in its own diagonal. We start by labeling $p_1$ with $1$. Then, we add $p_2$, label it with $-2$, and toggle its value to $3 = 1-(-2)$. We add $p_3$, label it with $-3$, and toggle its value to $3-(-3)$. We add $p_4$, label it with $-4$, and toggle its value to $3-(-4)$. 
    
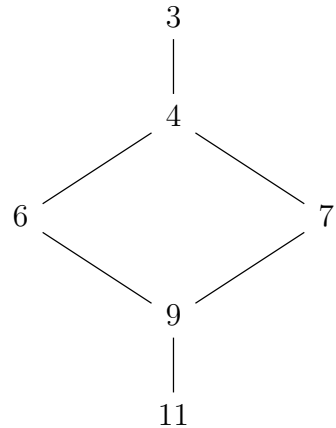
\begin{center}
        \begin{tikzcd}
	& {\text{Step 5}} &&& {\text{Step 6}} \\
	& 1 &&& 3 \\
	& 4 &&& 4 \\
	6 && 7 & 6 && 7 \\
	& 9 &&& 9 \\
	&&&& 11
	\arrow[no head, from=2-2, to=3-2]
	\arrow[no head, from=3-2, to=4-1]
	\arrow[no head, from=3-5, to=2-5]
	\arrow[no head, from=3-5, to=4-6]
	\arrow[no head, from=4-1, to=5-2]
	\arrow[no head, from=4-3, to=3-2]
	\arrow[no head, from=4-4, to=3-5]
	\arrow[no head, from=4-6, to=5-5]
	\arrow[no head, from=5-2, to=4-3]
	\arrow[no head, from=5-5, to=4-4]
	\arrow[no head, from=6-5, to=5-5]
        \end{tikzcd}
    \end{center}
    When we add $p_5$, we set its value to $-2$. We toggle its value to $7-(-2) = 9$, but we also toggle the other element in its diagonal from $3$ to $1+6-3 = 4$. 
    Finally, when we add $p_6$, we set its value to $-2$. Its value gets toggled to $9-(-2)$, and the value of its other diagonal element gets toggled from $1$ to $4-1 = 3$. 
\end{example}

\begin{remark}
    When the poset element $y$ exists in the toggle operation,
    we have $s'_p = s_x + s_y - s_p$,
    which if $s_x \leq s_p \leq s_y$
    satisfies $s_x \leq s'_p \leq s_y$.
    The only step of RSK when $s_y$ does not exist
    is when we toggle $c$,
    in which case $s'_c = s_x + t_c \geq s_x$.
    Thus $\RSK_P$ maps $\Ff_P$ into $\Ff'_P.$
    Also, toggle operations are reversible,
    and since $\RSK_P$ consists of repeated applications of toggle operations,
    $\RSK_P$ is indeed a bijection $\Ff_P \to \Ff'_P$.
\end{remark}

We now verify that $\RSK_P$ is well-defined.

\begin{prop}
    \label{prop:rsk_well_def}
    The output of $\RSK_P$ does not depend 
    on the choice of linear extension of $P$. 
\end{prop}

\begin{proof}
    Let $P' = P \sqcup \{a, b\}$ be a \dc poset,
    with $P$ an upper set and $a$, $b$ distinct, non-adjacent elements.
    Then for an ordering $\{p_i\}_{i = 1}^n$ of the elements of $P$,
    appending $a$ and $b$ to the end in either order
    gives a valid insertion order for $P'$.
    We wish to show that the two orders give the same result,
    which will prove the claim
    because any linear extension can be obtained from any other
    by a sequence of such transpositions
    (see, for example, \cite[Lemma~2.3]{ayyer2014combinatorial}).
    In general,
    for a labelled \dc poset $Q' = Q \sqcup \{c\}$ with $Q$ an upper set,
    the computation of $\RSK_{Q'}([t_p]_{p \in Q'})$
    from $\RSK_Q([t_p]_{p \in Q})$
    depends only on the labels of $\RSK_Q([t_p]_{p \in Q})$
    at elements in $D(c)$ and adjacent diagonals,
    and only changes the values at elements in $D(c)$.
    Since $D(a)$ and $D(b)$ are disjoint
    by \Cref{prop:diag_props}\ref{dprop:chain}
    and non-adjacent by \Cref{prop:diag_props}\ref{dprop:min},
    the order of addition of $a$ and $b$ does not matter.
\end{proof}

\section{Properties of RSK}
\label{sec:rsk_prop}

In this section, we give several properties of RSK for \dc posets,
focusing particularly on the relationship between RSK and hook vectors.
Our results, namely \Cref{prop:bijec_d_volume}
and \Cref{prop:bijec_d_hl},
will be crucial for our proof of the multivariate hook length formula.

Viewing $\Ff_P$ and $\Ff'_P$ as subsets (cones, in fact) of $\RR^n$,
we observe that $\RSK_P$ is a piecewise-linear map from $\RR^n$ to $\RR^n$.

\begin{prop}
    \label{prop:bijec_d_volume}
    Let $P$ be a \dc poset with $n$ elements.
    Then the piecewise linear map
    $\RSK_P \colon \RR^n \rightarrow \RR^n$
    is volume-preserving.
\end{prop}

\begin{proof}
    It suffices to prove that $\RSK_P$ is volume-preserving
    on regions on $\RR^n$ where it is linear.
    Since $\RSK_P$ is a composition of toggle operations,
    it suffices to prove that a toggle at $c \in P$
    effects a volume-preserving map $\RR^n \to \RR^n$.
    The matrix describing the linear map induced by a
    toggle at $c$
    has all $1$s on the diagonal
    except for a $-1$ at position $(c, c)$,
    some number of $1$s in the $c$th row,
    and $0$s everywhere else.
    This has determinant $-1$, and thus is volume-preserving.
\end{proof}

The crux of the geometric proof of the usual hook length formula
was the following lemma
(see \cite[Proposition~13]{hopkins2014rsk},
generalizing a computation of \cite[Section~4]{pak2001hook}),
which shows that $\RSK$ actually sends one polytope into the other.

\begin{lemma}\label{claim:diagonalsums}
    If $[a_{ij}]_{(i,j) \in \lambda}$ is a point of $P_{\text{fillings}}$
    and $[r_{ij}]_{(i,j) \in \lambda}$ is its image under RSK, then,
    denoting the hook length of cell $(i, j)$ by $h_{ij}$,
    we have
    \[\sum_{(i,j) \in \lambda} r_{ij} = \sum_{(i,j) \in \lambda}h_{ij}a_{ij}.\]
\end{lemma}

To prove the multivariate generalization of the hook length formula,
we will need a stronger fact involving the diagonals of $P$.
First, we introduce some notation.

\begin{definition}
    Fix $[t_p]_{p \in P} \in \Ff_P$,
    and let $[s_p]_{p \in P} = \RSK_P([t_p]_{p \in P})$.
    For each diagonal $D$ of $P$, define the \emph{diagonal sum}
    $S_P(D) = \sum_{p \in D} s_p$.
\end{definition}
 
With this notation,
our strengthening of \Cref{claim:diagonalsums} is as follows.

\begin{prop}
    \label{prop:bijec_d_hl}
    Let $P$ be a \dc poset.
    Fix $[t_p]_{p \in P} \in \Ff_P$,
    and let $[s_p]_{p \in P} = \RSK_P([t_p]_{p \in P})$.
    Then for every diagonal $D$ of $P$,
    we have $S_P(D) = \sum_{p \in P} h_P^{(D)}(p) t_p$.
\end{prop}

\begin{remark}
    For Young diagrams,
    summing over all diagonals gives \Cref{claim:diagonalsums}.
\end{remark}

Proving \Cref{prop:bijec_d_hl} is the most technical step of our proof,
and will require a couple intermediary results.

Since we may choose any linear extension of $P$ when computing $\RSK_P$,
it will be useful for us to choose a particularly nice linear extension.

\begin{definition}
    A linear extension $p_1 > p_2 > \dots > p_n$ of $P$ is called
    a \emph{stable insertion order}
    if whenever $p_i$ is a bottom element of a $d$-interval,
    the side elements of that $d$-interval are not neck elements
    of any other $d$-interval in $P^{(i)} = \{p_1, \dots, p_i\}$.
\end{definition}

Informally, this condition says that we should finish inserting
the tail of a $d$-interval that we have already started before
moving on to $d$-intervals that are lower down in $P$.

\begin{example}
    A stable insertion order for the \dc poset from
    \Cref{ex:d_comp} is shown below.
    
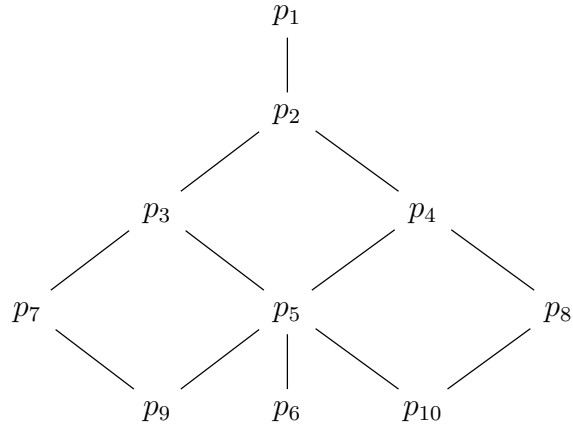
\begin{center}
        \begin{tikzcd}
            && p_1 \\
            && p_2 \\
            & p_3 && p_4 \\
            p_7 && p_5 && p_8 \\
            & p_9 & p_6 & p_{10} \\
            \arrow[no head, from=1-3, to=2-3]
            \arrow[no head, from=2-3, to=3-2]
            \arrow[no head, from=2-3, to=3-4]
            \arrow[no head, from=3-2, to=4-1]
            \arrow[no head, from=3-2, to=4-3]
            \arrow[no head, from=3-4, to=4-3]
            \arrow[no head, from=3-4, to=4-5]
            \arrow[no head, from=4-1, to=5-2]
            \arrow[no head, from=4-3, to=5-2]
            \arrow[no head, from=4-3, to=5-4]
            \arrow[no head, from=4-5, to=5-4]
            \arrow[no head, from=4-3, to=5-3]
        \end{tikzcd}
    \end{center}
    Although $p_3$ and $p_4$ are neck elements of $d$-intervals in $P$,
    they are not in $P^{(5)}$ and $P^{(6)}$.
    Also, $p_5$ is not a neck element of $P$,
    and thus is not a neck element of $P^{(9)}$ or $P^{(10)}$.
\end{example}

\begin{lemma}\label{prop:stable_order}
    Every \dc poset has at least one stable insertion order. 
\end{lemma}

\begin{proof}
    We proceed by strong induction under poset inclusion;
    the unique ordering on a singleton $P$
    establishes the base case.
    Now suppose $P$ has more than one element.
    If $P = P' \sqcup \{c\}$ for minimal element $c$ of $P$
    that is not in any $d$-interval,
    then appending $c$ to the end of a stable ordering for $P'$
    gives a stable ordering for $P$.

    Now suppose that all minimal elements 
    of $P$ are in $d$-intervals.
    We may identify each maximal $d$-interval of $P$
    with the top element of its diamond;
    this gives an injective map into $P$
    because by \Cref{prop:d_int_closure}
    distinct maximal $d$-intervals must have distinct diamonds. Thus, this injection induces a partial ordering on the maximal $d$-intervals of $P$.
    Let $[c, p]$ be minimal under this ordering.
    First, we note that the side elements of $[c, p]$
    cannot be neck elements of other $d$-intervals,
    since then those $d$-intervals would contradict
    the minimality of $[c, p]$.

    Now we claim that $c$ is minimal in $P$.
    If not, $c$ covers some element $c'$ of $P$.
    Since $[c, p]$ is a maximal $d$-interval (under inclusion),
    $\{c'\} \cup [c, p]$ cannot be $d^-$-convex. 
    Thus, $c'$ must be covered by an element $y$ of $P$ not in $[c, p]$,
    and since $\{c', c, y\}$ is $d_3^-$-convex,
    there must exist $x$ covering $c$ and $y$.
    Since $c$ is in the tail of $[c, p]$, by \Cref{prop:d_int_closure},
    $x$ must be in $[c, p]$,
    and in particular is a tail element or a side element.
    Either way, the top element of the $d_3$-interval $[c', x]$
    is dominated by the top element of the diamond of $[c, p]$,
    contradicting the minimalty assumption on $[c, p]$.

    Now by the induction hypothesis,
    appending $c$ to the end of a stable insertion order
    for $P \setminus \{c\}$
    gives a stable insertion order.
\end{proof}

To prove \Cref{prop:bijec_d_hl},
we need to understand how diagonal sums are changed
by insertion in RSK.
If $P$ is an upper set of \dc poset $P'$,
then by \Cref{prop:diag_props}\ref{dprop:up_id} diagonals of $P'$
correspond to (possibly empty) diagonals of $P$
and by \Cref{prop:diag_props}\ref{dprop:up_adj} a pair of diagonals
is adjacent in $P$
if and only if it is adjacent in $P'$.
Thus for a filling $[t_p]_{p \in P'}$ of $P'$
and corresponding $[s_p]_{p \in P} = \RSK_P([t_p]_{p \in P})$
and $[s'_p]_{p \in P'} = \RSK_{P'}([t_p]_{p \in P'})$
we may consider both $S_P(D)$ and $S_{P'}(D)$,
where the sum in the former is only over $p \in P$.
The following lemma relates the two quantities.

\begin{lemma}
    \label{lem:diag_sum_adj}
    Let $P' = P \sqcup \{c\}$ be \dc
    with minimal element $c$.
    Let $[t_p]_{p \in P'}$ be a filling of $P'$.
    Let $[s_p]_{p \in P} = \RSK_P([t_p]_{p \in P})$
    and $[s'_p]_{p \in P'} = \RSK_{P'}([t_p]_{p \in P'})$.
    Let $\Dd$ denote the set of diagonals of $P'$ which are adjacent to $D(c)$.
    Then $S_P(D(c)) + S_{P'}(D(c)) = t_c + \sum_{D \in \Dd} S_P(D)$.
\end{lemma}

\begin{proof}
    For convenience, we set $s_c = -t_c$.
    Redefining $S_P$ accordingly for the rest of this proof,
    we wish to show that $S_P(D(c)) + S_{P'}(D(c)) = \sum_{D \in \Dd} S_P(D)$.
    Also, we assume that all $s_p$ are distinct;
    the case when some $s_p$ coincide will then follow from continuity
    of the piecewise-linear map $\RSK_{P'}$.
    By \Cref{prop:rsk_well_def} we may assume that $c$ was inserted last in RSK,
    so that for each $d \in D(c)$, we have
    $s'_d + s_d = s'_{m_d} + s'_{n_d}$,
    where $m_d$ and $n_d$ are
    the elements covering and covered by $d$,
    respectively,
    with maximal $s'_{m_d}$ and minimal $s'_{n_d}$
    (if $m_d$ or $n_d$ does not exist,
    we set the corresponding $s'_\bullet$ to be $0$).
    By definition, we have $D(m_d)$, $D(n_d) \in \Dd$ for $d \in D(c)$.
    We wish to show that for each element $p$ of $P$
    with $D(p) \in \Dd$,
    the term $s_p$ occurs exactly once
    in the formal sum
    $S_P(D(c)) + S_{P'}(D(c)) = \sum_{d \in D(c)} (s_d + s'_d)
    = \sum_{d \in D(c)} (s_{m_d} + s_{n_d})$.

    By \Cref{prop:diag_props}\ref{dprop:adj},
    the element $p$ is adjacent to some element of $D(c)$.
    We claim that $p$ cannot be covered by two distinct elements of $D(c)$,
    and cannot cover two distinct elements of $D(c)$.
    Suppose first that $p$ covers two elements $c'$ and $c''$ of $D(c)$.
    Since $D(c)$ is a chain by \Cref{prop:diag_props}\ref{dprop:chain},
    we may assume $c' < c''$.
    Because $c'$ is not the largest element in its diagonal,
    it is the tail element of a $d$-interval;
    then by \Cref{prop:d_int_closure},
    $p$ is in the $d$-interval $[c', c'^\uparrow]$, so $c'^\uparrow > p$.
    But $c'' \geq c'^\uparrow$,
    which contradicts $p$ covering $c''$.
    Suppose now that $p$ is covered by two elements $c'$ and $c''$ of $D(c)$,
    and assume $c' > c''$.
    Because $c'$ is not the smallest element in its diagonal,
    it is the neck element of a $d$-interval;
    then by \Cref{prop:d_int_closure},
    $p$ is in the $d$-interval $[c'^\downarrow, c']$,
    and so $p > c'^\downarrow$.
    Since $c'' \leq c'^\downarrow$,
    this contradicts $p$ being covered by $c''$.

    We are now ready to prove that $s_p$ occurs exactly once
    in the formal sum $\sum_{d \in D(c)} (s_{m_d} + s_{n_d})$.
    We consider a few cases.
    \begin{itemize}
        \item The element $p$ is adjacent to only one element $d$ of $D(c)$,
            and $p$ covers $d$.
            We claim that $p$ is the only element covering $d$;
            this will complete the proof since then
            $p = m_d$ and $p$ is adjacent to no other element of $D(c)$.
            If there were another element $p'$ covering $d$,
            then $\{p, p', d\}$ would be $d_3^-$-convex,
            so there would be $d' \in D(c)$
            such that $\{d, p, p', d'\}$ is a diamond.
            This contradicts the assumption that $p$ is adjacent
            to only one element of $D(c)$.
        \item The element $p$ is adjacent to only one element $d$ of $D(c)$,
            and $p$ is covered by $d$.
            We claim that $p$ is the only element covered by $d$;
            this will complete the proof since then
            $p = n_d$ and $p$ is adjacent to no other element of $D(c)$.
            Since $d$ covers an element of $P$,
            we have $d \neq c$, so $d^\downarrow$ exists.
            Then $d$ cannot cover any elements outside of
            the $d$-interval $[d^\downarrow, d]$,
            so if $d$ covers an element $p' \neq p$ of $P$,
            the interval $[d^\downarrow, d]$ must be a diamond
            with side elements $p$ and $p'$.
            But then $p$ is adjacent to $d^\downarrow$, a contradiction.
        \item The element $p$ covers $d$ and is covered by $d'$ in $D(c)$.
            The smallest possible distance
            between the maximal and minimal elements
            of a $d$-interval is $2$,
            so there cannot be elements of $D(c)$ between $d$ and $d'$.
            Thus $[d, d']$ is a diamond, one of whose side element is $p$;
            let the other be $p'$.
            If $s_p < s_{p'}$,
            then $m_d = p$ and $n_{d'} = p'$;
            if $s_p > s_{p'}$,
            then $m_d = p'$ and $n_{d'} = p$.
            Either way, $s_p$ occurs exactly once in the formal sum
            $\sum_{d \in D(c)}(s_{m_d} + s_{n_d})$.
    \end{itemize}
    Since these cases are exhaustive,
    this completes the proof.
\end{proof}

We may now prove \Cref{prop:bijec_d_hl}.

\begin{proof}[Proof of \Cref{prop:bijec_d_hl}]
    We proceed by strong induction under poset inclusion;
    the statement for singleton $P$ is immediate
    because $P$ has only one element and one diagonal.
    Let $P'$ be a \dc poset
    with filling $[t_p]_{p \in P'}$,
    and let $c$ be the last element
    of some stable insertion order on $P'$,
    which exists by \Cref{prop:stable_order}.
    Let $P$ be such that $P' = P \sqcup \{c\}$,
    and assume that the statement of the proposition is true for $P$,
    that is, that for each diagonal $D$, we have
    \[
        S_P(D) = \sum_{p \in P} h_P^{(D)}(p) t_p.
    \]
    We wish to obtain the analogous equality for $P'$,
    whence the result will follow by induction.
    We consider the cases when $D = D(c)$ and when $D \neq D(c)$ separately.
    Within each case, we study hook vectors of individual elements $p$
    and consider three subcases:
    \begin{itemize}
        \item the element $p$ is a neck element of both $P'$ and $P$;
        \item the element $p$ is a neck element of $P'$, but not of $P$;
        \item the element $p$ is not a neck element of $P'$.
    \end{itemize}
    Since $d$-intervals of $P$ are $d$-intervals in $P'$,
    these subcases are exhaustive.
    
    \textit{Case 1:} $D \neq D(c)$.
            Since only labels in $D(c)$ are toggled
            when $c$ is inserted in RSK,
            we have $S_P(D) = S_{P'}(D)$.
            We will prove that for all $p \in P'$,
            we have $h_P^{(D)}(p) = h_{P'}^{(D)}(p)$,
            from which it will follow that
            $S_{P'}(D) = \sum_{p \in P'} h_{P'}^{(D)}(p) t_p$
            as needed.

            We proceed by strong induction on $p$,
            with the base case being minimal elements of $P'$;
            for such $p \neq c$
            we have $h_P^{(D)}(p) = \delta_{D(p), D} = h_{P'}^{(D)}(p)$,
            and for $c$ we have
            $H_P^{(D)}(c) = 0 = \delta_{D(c), D} =  h_{P'}^{(D)}(c)$.
            Now suppose the claim is true for all elements strictly dominated
            by a given $p$.
            We consider several cases.
            \begin{itemize}
                \item The element $p$ is a neck element of both $P'$ and $P$.
                    Then there is an element $q \in P$
                    such that $[q, p]$ is a $d$-interval
                    with side elements $w$ and $z$.
                    Then $\bh_P(p) = \bh_P(w) + \bh_P(z) - \bh_P(q)$
                    and
                    $\bh_{P'}(p) = \bh_{P'}(w) + \bh_{P'}(z) - \bh_{P'}(q)$,
                    so
                    \begin{align*}
                        h_{P'}^{(D)}(p)
                        &= h_{P'}^{(D)}(w) + h_{P'}^{(D)}(z)
                            - h_{P'}^{(D)}(q) \\
                        &= h_P^{(D)}(w) + h_P^{(D)}(z)
                            - h_P^{(D)}(q) \\
                        &= h_P^{(D)}(p)
                    \end{align*}
                    by the induction hypothesis.
                \item The element $p$ is a neck element of $P'$ but not of $P$,
                    so $[c, p]$ is a $d$-interval of $P'$.
                    Let $w$, $z$ be the side elements of $[c, p]$.
                    Since the insertion order is stable,
                    $w$ and $z$ are not neck elements in $P'$.
                    Thus, $\bh_P(w)$ and $\bh_P(z)$ are the indicator vectors
                    of the sets of elements dominated by $w$ and $z$,
                    respectively, in $P$
                    and $\bh_{P'}(w) = \bh_P(w) + \bone_c$
                    and $\bh_{P'}(z) = \bh_P(z) + \bone_c$.
                    Then $\bh_{P'}(p) = \bh_{P'}(w) + \bh_{P'}(z) - \bone_c
                    = \bh_P(w) + \bh_P(z) + \bone_c$.

                    Letting $T$ be the set of elements of $P$
                    dominated by both $w$ and $z$
                    and $N$ be the set of elements of $P$
                    dominated by $p$ but not $w$ or $z$,
                    we have
                    $\bh_P(p) = \bh_P(w) + \bh_P(z) + \bone_N - \bone_T$.
                    Now $N$ is the neck of $[c, p]$
                    and $T \sqcup \{c\}$ is the tail,
                    and each tail element of $[c, p]$
                    corresponds bijectively to a neck element
                    in the same diagonal.
                    Thus, $\bone_N - \bone_T = \bone_c$,
                    giving $\bh_P(p) = \bh_{P'}(p)$
                    as needed.
                \item The element $p$ is not a neck element of $P'$.
                    Then $p$ is not a neck element of $P$ either.
                    By definition, $\bh_P(p)$ and $\bh_{P'}(p)$
                    are indicator vectors of the sets of elements
                    nonstrictly dominated by $p$ in $P$ and $P'$, respectively.
                    Thus, $\bh_{P'}(p) = \bh_P(p) + \bone_c$
                    or $\bh_{P'}(p) = \bh_P(p)$
                    depending on whether $p$ dominates $c$,
                    and since $c \notin D$,
                    we have $h_{P'}^{(D)}(p) = h_P^{(D)}(p)$ either way.
            \end{itemize}
    \textit{Case 2:} $D = D(c)$. Let $\Dd$ be the set of diagonals of $P'$
            which are adjacent to $D(c)$.
            Then $S_P(D(c)) + S_{P'}(D(c)) = t_c + \sum_{D' \in \Dd} S_P(D')$
            by \Cref{lem:diag_sum_adj}.
            We have $h_{P'}^{(D)}(c) = 1$
            and we claim that for $p \in P$, we have
            \[
                h_P^{(D)}(p) + h_{P'}^{(D)}(p)
                = \sum_{D' \in \Dd} h_P^{(D)}(p).
            \]
            The induction step will then follow because
            \begin{align*}
                S_{P'}(D) &= t_c - S_P(D) + \sum_{D' \in \Dd} S_P(D') \\
                &= t_c + \sum_{p \in P} \left( -h_P^{(D)}(p)
                    + \sum_{D' \in \Dd} h_P^{(D')}(p) \right) t_p \\
                &= t_c + \sum_{p \in P} h_{P'}^{(D)}(p) t_p \\
                &= \sum_{p \in P'} h_{P'}^{(D)}(p) t_p.
            \end{align*}
            We again proceed by strong induction on $p$.
            When $p$ is minimal in $P$,
            we have $h_P^{(D)}(p) = \delta_{D(p), D} = 0$,
            and $\bh_{P'}(p)$ is $\bone_p$ or $\bone_p + \bone_c$
            depending on whether $p$ dominates $c$ or not;
            since $p$ dominates $c$ if and only if
            it is in a diagonal adjacent to $D$,
            this proves the claim for minimal $p$.
            Now suppose the claim is true for all elements strictly dominated
            by a given $p \neq c$.
            We consider the same cases as before.
            \begin{itemize}
                \item The element $p$ is a neck element of both $P'$ and $P$.
                    Then there is an element $q \in P$
                    such that $[q, p]$ is a $d$-interval
                    with incomparable elements $w$ and $z$.
                    Then $\bh_P(p) = \bh_P(w) + \bh_P(z) - \bh_P(q)$
                    and
                    $\bh_{P'}(p) = \bh_{P'}(w) + \bh_{P'}(z) - \bh_{P'}(q)$,
                    so
                    \begin{align*}
                        h_{P'}^{(D)}(p) + h_P^{(D)}(p)
                        &= h_{P'}^{(D)}(w) + h_{P'}^{(D)}(z)
                            - h_{P'}^{(D)}(q) \\
                        &\phantom{=}
                            + h_P^{(D)}(w) + h_P^{(D)}(z)
                            - h_P^{(D)}(q) \\
                        &= \sum_{D' \in \Dd}
                            \left( h_P^{(D')}(w) + h_P^{(D')}(z)
                            - h_P^{(D')}(q) \right) \\
                        &= \sum_{D' \in \Dd} h_P^{(D')}(p)
                    \end{align*}
                    by the induction hypothesis.
                \item The element $p$ is a neck element of $P'$ but not of $P$,
                    so $[c, p]$ is a $d$-interval of $P'$.
                    Let $w$, $z$ be the side elements of $[c, p]$.
                    Since the insertion order is stable,
                    $w$ and $z$ are not neck elements in $P'$.
                    Thus, $\bh_P(w)$ and $\bh_P(z)$ are indicator vectors
                    of the sets of elements
                    dominated by $w$ and $z$, respectively, in $P$
                    and $\bh_{P'}(w) = \bh_P(w) + \bone_c$
                    and $\bh_{P'}(z) = \bh_P(z) + \bone_c$.
                    Then $\bh_{P'}(p) = \bh_{P'}(w) + \bh_{P'}(z) - \bone_c
                    = \bh_P(w) + \bh_P(z) + \bone_c$.
                    Since $w$ and $z$ do not dominate $c^\uparrow = p$,
                    we have $h_P^{(D)}(w) = h_P^{(D)}(z) = 0$,
                    whence $h_{P'}^{(D)}(p) = 1$.

                    Letting $T$ be the set of elements of $P$
                    dominated by both $w$ and $z$
                    and $N$ be the set of elements of $P$
                    dominated by $p$ but not $w$ or $z$,
                    we have
                    $\bh_P(p) = \bh_P(w) + \bh_P(z) + \bone_N - \bone_T$.
                    Since 
                    $c^\uparrow = p$
                    we have $h_P^{(D)}(w) = h_P^{(D)}(z) = |T \cap D| = 0$
                    and $|N \cap D| = 1$,
                    whence $h_P^{(D)}(p) = 1$.

                    It follows that $h_P^{(D)}(p) + H_{P'}^{(D)}(p) = 2$.
                    Also, $\bh_P(p)$ is the indicator vector of
                    $[c, p] \setminus \{c\}$
                    and by \Cref{prop:diag_props}\ref{dprop:adj},
                    the only elements of $[c, p] \setminus \{c\}$
                    in a diagonal adjacent to $D$
                    are those which are adjacent to $c$ or $p$.
                    There are two such elements,
                    giving $\sum_{D' \in \Dd} h_P^{(D')}(p) = 2$,
                    as needed.
                \item The element $p$ is not a neck element of $P'$.
                    By definition, $\bh_P(p)$ and $\bh_{P'}(p)$
                    are indicator vectors of the sets of elements
                    nonstrictly dominated by $p$ in $P$ and $P'$, respectively.
                    We now consider two cases
                    based on whether $p$ dominates $c$.

                    \textit{Subcase:}
                    The element $p$ does not dominate $c$.
                    Then $p$ cannot dominate any element of $D$,
                    so $h_P^{(D)}(p) + h_{P'}^{(D)}(p) = 0$.
                    Now suppose for the sake of contradiction
                    that there exists some $d \in D' \in \Dd$
                    dominated by $p$.
                    By \Cref{prop:diag_props}\ref{dprop:adj},
                    the element $d$ is adjacent to some element $c'$
                    of $D(c)$,
                    and by assumption cannot dominate $c'$.
                    Thus $c'$ covers $d$.
                    It follows by \Cref{prop:d_int_closure}
                    that $d$ is in the $d$-interval
                    $[c'^\downarrow, c']$,
                    whence $p$ dominates $c'^\downarrow$ and hence $c$,
                    a contradiction.
                    So $\sum_{D' \in \Dd} h_P^{(D')}(p) = 0$,
                    as desired.

                    \textit{Subcase:}
                    The element $p$ dominates $c$.
                    Then $h_P^{(D)}(p) + h_{P'}^{(D)}(p)
                    = 2h_P^{(D)}(p) + 1$.
                    Let $c = c_0 < \dots < c_k$ be the elements of
                    $D$ dominated by $p$ in $P'$,
                    so that $k = h_P^{(D)}(p)$.
                    Each $d$-interval $[c_i, c_{i + 1}]$
                    contains two elements in some $D' \in \Dd$.
                    Since $p$ is not a neck element in $P'$,
                    we have $p > c_k$ with strict inequality.
                    Also, $c_k$ is covered by exactly one element
                    which is in the principal ideal generated by $p$,
                    since otherwise $c_k^\uparrow$ would exist
                    and thus need to be be in
                    the principal ideal generated by $p$.
                    This gives $2k + 1$ elements
                    counting towards $\sum_{D' \in \Dd} h_P^{(D')}(p)$.
                    Also, every element counting towards this sum
                    has been accounted for;
                    if $d \in D'$ for $D' \in \Dd$
                    is dominated by $p$
                    then $d$ is adjacent to $c' \in C$ by
                    \Cref{prop:diag_props}\ref{dprop:adj},
                    and if $c'$ is not dominated by $p$ then $d$
                    is either a side element of $[c'^\downarrow, c']$,
                    in which case $d$ covers $c'^\downarrow = c_k$
                    and thus has been counted,
                    or $d$ is a neck element of $[c'^\downarrow, c']$
                    which is dominated by $p$
                    and hence covered by some element other than $c'$,
                    contradicting \Cref{prop:d_int_closure}.
                    It follows that
                    $\sum_{D' \in \Dd} h_P^{(D')}(p) = 2k + 1
                    = 2 h_P^{(D)}(p) + 1$, as desired.
            \end{itemize}
    As remarked earlier, these cases are exhaustive, so the proof is complete.
\end{proof}

\begin{remark}
    It follows from the analysis of the evolution of $\bh_P(p)$
    over the RSK insertion process
    in the proof of \Cref{prop:bijec_d_hl}
    that $\bh_P(p)$ has nonnegative entries.
    Thus we may interpret hook vectors
    as multisets of diagonals.
\end{remark}

\section{Proof of multivariate hook length formula}
\label{sec:proof}

From this point,
the proof of \Cref{thm:mainthm} follows that of Pak~\cite{pak2001hook}.

\begin{proof}[Proof of \Cref{thm:mainthm}]
    For a \dc poset $P$ with $n$ elements, we define two polytopes
    $P_\text{fillings} = P_\text{fillings}(\bx)$
    and $P_\text{RPP} = P_\text{RPP}(\bx)$
    depending on positive real parameters $x_D$
    corresponding to diagonals of $P$
    in the $n$-dimensional real space
    of vectors with entries indexed by elements of $P$.
    The polytope $P_\text{fillings}$ will be a subset of $\Ff_P$,
    and the polytope $P_\text{RPP}$ will be a subset of $\Ff'_P$.
    We will find that $\RSK_P$ gives a volume-preserving map
    between $P_\text{fillings}$ and $P_\text{RPP}$
    which yields our result.

    The polytope $P_\text{fillings}$
    consists of all vectors $[t_p]_{p \in P}$ satisfying:
    \begin{itemize}
        \item $t_p \geq 0$ for all $p \in P$;
        \item $\sum_{p \in P} H_p(\bx) t_p \leq 1$.
    \end{itemize}
    Since $P_\text{fillings}$ is a rescaling of the standard $n$-simplex,
    the volume of $P_\text{fillings}$ is
    \[
        \frac1{n!} \prod_{p \in P} \frac1{H_p(\bx)}.
    \]

    The polytope $P_\text{RPP}$
    consists of all vectors $[s_p]_{p \in P}$ satisfying:
    \begin{itemize}
        \item $s_p \geq 0$ for all $p \in P$;
        \item $s_p \geq s_q$ whenever $p \leq_P q$;
        \item $\sum_{p \in P} x_{D(p)} s_p \leq 1$.
    \end{itemize}
    This is the union of polytopes
    defined for each linear extensions $T$ of $P$ given by $p_1 > \dots > p_n$
    as the set of all vectors satisfying:
    \begin{itemize}
        \item $0 \leq s_{p_1} \leq \dots \leq s_{p_n}$;
        \item $\sum_{i = 1}^n x_{D(p_i)} s_{p_i} \leq 1$
    \end{itemize}
    We prove by induction on the dimension of space
    that such a polytope has volume $\frac1{n!}\weight(T)$.
    The only vertex with nonzero $s_{p_1}$ coordinate
    is given by $s_{p_1} = \dots = s_{p_n}$
    and $\sum_{i = 1}^n x_{D(p_i)}s_{p_i} = 1$,
    and thus is located at $(\sum_{i = 1}^n x_{D(p_i)})^{-1} \ind$,
    where $\ind$ is the vector of all $1$s.
    Inductively this gives a volume of
    \[
        \frac1{n!} \prod_{i = 1}^n
        \left( \sum_{j = i}^n x_{D(p_i)} \right)^{-1}
        = \frac1{n!} \weight(T),
    \]
    as claimed.
    Thus $P_\text{RPP}$ has volume
    \[
        \frac1{n!} \sum_{T} \weight(T),
    \]
    where the sum is over linear extensions $T$ of $P$.

    The map $\RSK_P$ is volume-preserving and bijective
    between $\Ff_P$ and $\Ff'_P$.
    We claim that $\RSK_P$ gives a bijection
    between $P_\text{fillings}$ and $P_\text{RPP}$.
    Let $[t_p]_{p \in P} \in P_\text{fillings}$,
    and let $[s_p]_{p \in P} = \RSK_P([t_p]_{p \in P})$.
    Let $\Dd$ be the set of diagonals of $P$.
    By \Cref{prop:bijec_d_hl}, we have
    \begin{align*}
        \sum_{p \in P} x_{D(p)} s_p
        &= \sum_{D \in \Dd} x_D S_P(D) \\
        &= \sum_{D \in \Dd} \sum_{p \in P} x_D h_P^{(D)}(p) t_p \\
        &= \sum_{p \in P} t_p \sum_{D \in \Dd} x_D h_P^{(D)}(p) \\
        &= \sum_{p \in P} H_p(\bx) t_p,
    \end{align*}
    so $[t_p]_{p \in P} \in P_\text{fillings}$
    if and only if $[s_p]_{p \in P} \in P_\text{RPP}$,
    as desired.

    It follows that
    \[
        \sum_T \weight(T) = \prod_{p \in P} \frac1{H_p(\bx)},
    \]
    completing the proof.
\end{proof}

\bibliography{main}
\bibliographystyle{alpha}

\newpage

\appendix

\section{Toggle operations and RSK}
\label{subsec:rsk_sketch}

We refer the readers to \cite{stanley2023enumerative}
for the background on RSK.
We recall that RSK is a bijective
map between $n \times n$ matrices of nonnegative integers,
which can be viewed as fillings of a square Young diagram,
and pairs of semistandard Young tableaux of the same shape.

\begin{example}\label{ex:classic-RSK}
    Consider the following filling of $\lambda = (3,3,3)$:
    \[ \;
    \begin{ytableau} 
        1 & 0 & 2 \\
        0 & 2 & 0 \\
        1 & 1 & 0
        \end{ytableau}.
    \]
    Its corresponding two-line array is
    \[ \left(\begin{matrix}
        1 & 1 & 1 & 2 & 2 & 3 & 3 \\
        1 & 3 & 3 & 2 & 2 & 1 & 2
    \end{matrix}\right). \]
    After insertion, we obtain two SSYTs
    \begin{align*}
        P &= \ctab{\begin{ytableau} 
            1 & 1 & 2 & 2 \\
            2 & 3 \\
            3
        \end{ytableau}},
        &
        Q &= \ctab{\begin{ytableau} 
            1 & 1 & 1 & 3 \\
            2 & 2 \\
            3
        \end{ytableau}}.
    \end{align*}
\end{example}

On the other hand, pick any ordering of the squares of $\lambda$
$(i_1,j_1), (i_2,j_2), \ldots$ such that $(i,j)$ appears before
$(i+1,j)$ and $(i,j+1)$ for all $i,j$. Let $(t_{i,j})$ be any filling
of $\lambda$ with nonnegative integers. One can construct a RPP of
shape $\lambda$ with entries $(p_{i,j})$ recursively as follows.

\begin{enumerate}
    \item If $\lambda = (1)$, set $p_{1,1}  = t_{1,1}$.
    \item Suppose we have built a RPP with the squares
        $(i_1,j_1),\ldots,(i_{\ell-1},j_{\ell-1})$,
    \begin{enumerate}
        \item set $p_{i_\ell,j_\ell} = \max(p_{i-1,j},p_{i,j-1}) + t_{i_\ell,j_\ell}$,
        \item for any other square $(i,j)$ in the same diagonal as
            $(i_{\ell},j_{\ell})$, toggle this square by replacing $p_{i,j}$
            with
            \[
                \max(p_{i-1,j},p_{i,j-1}) + \min(p_{i+1,j},p_{i,j+1})
                - p_{i,j}.
            \]
    \end{enumerate}
    If any of $p_{i-1,j},p_{i,j-1},p_{i+1,j},p_{i,j+1}$ is not in the RPP, we take the value to be $0$.
\end{enumerate}

\begin{example}
    Consider the same filling as in \Cref{ex:classic-RSK}, and order the
    squares as follows:
    \[
        \begin{ytableau} 
            1 & 2 & 3 \\
            4 & 6 & 8 \\
            5 & 7 & 9
        \end{ytableau}.
    \]
    Then, the above procedure is carried out as follows.
    \[
        \ctab
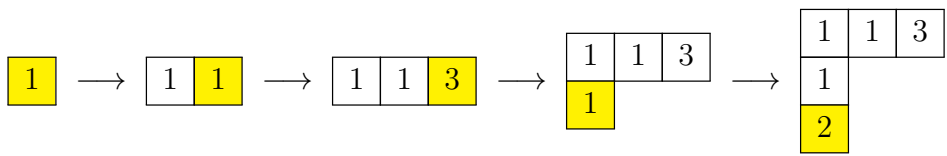
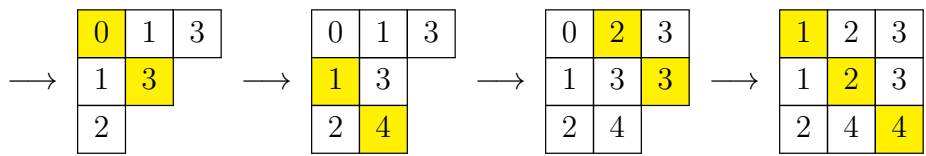
{\begin{ytableau} 
            *(yellow) 1
        \end{ytableau}}
        ~\longrightarrow~
        \ctab{\begin{ytableau} 
            1 & *(yellow) 1
        \end{ytableau}}
        ~\longrightarrow~
        \ctab{\begin{ytableau} 
            1 & 1 & *(yellow) 3
        \end{ytableau}}
        ~\longrightarrow~
        \ctab{\begin{ytableau} 
            1 & 1 & 3 \\
            *(yellow) 1
        \end{ytableau}}
        ~\longrightarrow~
        \ctab{\begin{ytableau} 
            1 & 1 & 3 \\
            1 \\
            *(yellow) 2
        \end{ytableau}}
    \]
    \[
        ~\longrightarrow~
        \ctab{\begin{ytableau} 
            *(yellow) 0 & 1 & 3 \\
            1 & *(yellow) 3 \\
            2
        \end{ytableau}}
        ~\longrightarrow~
        \ctab{\begin{ytableau} 
            0 & 1 & 3 \\
            *(yellow) 1 & 3 \\
            2 & *(yellow) 4
        \end{ytableau}}
        ~\longrightarrow~
        \ctab{\begin{ytableau} 
            0 & *(yellow) 2 & 3 \\
            1 & 3 & *(yellow) 3 \\
            2 & 4
        \end{ytableau}}
        ~\longrightarrow~
        \ctab{\begin{ytableau} 
            *(yellow) 1 & 2 & 3 \\
            1 & *(yellow) 2 & 3 \\
            2 & 4 & *(yellow) 4
        \end{ytableau}}
    \]
    In fact, the resulting RPP corresponds to the two SSYTs in
    \Cref{ex:classic-RSK}.
    The lower triangular part of this RPP is the Gelfand–Tsetlin pattern
    \[
        \begin{matrix}
            4 & & 2 & & 1 \\
            & 4 & & 1 & \\
            & & 2 & &
        \end{matrix}.
    \]
    There is a unique SSYT such that the squares containing $1,\ldots,i$
    form the shape given by the $i$th row (from bottom to top) of the
    Gelfand–Tsetlin pattern. This SSYT is the SSYT $P$ in
    \Cref{ex:classic-RSK}. Similarly, the upper triangular part of this
    RPP is the Gelfand–Tsetlin pattern
    \[ \begin{matrix}
        4 & & 2 & & 1 \\
        & 3 & & 2 & \\
        & & 3 & &
    \end{matrix}. \]
    The corresponding SSYT is the SSYT $Q$ in \Cref{ex:classic-RSK}.
\end{example}

\end{document}